\documentclass[11pt]{article}
\usepackage{amssymb}
\usepackage{amsthm}
\usepackage{amsmath}

\usepackage{url}
\usepackage{titling}
\usepackage{tikz-cd}

\usepackage{color}
\usepackage{geometry}               
\usepackage[parfill]{parskip}    % Activate to begin paragraphs with an empty line rather than an indent

\usepackage{graphicx}
\usepackage{amssymb}
\usepackage{epstopdf}
\usepackage{amsthm}

\newtheorem{thm}{Theorem}
\newtheorem{lemma}[thm]{Lemma}

\newtheorem{prop}[thm]{Proposition}

\newtheorem{definition}[thm]{Definition}

\renewcommand{\phi}{\varphi}

\renewcommand{\-}{^{-1}}

\newcommand{\N}{\mathbb{N}}

\title{The braid group injects in the virtual braid group}
\author{Robin Gaudreau}
\begin{document}

\maketitle

\begin{abstract} The virtual braid groups are generalizations of the classical braid groups.
This paper gives an elementary proof that the classical braid group injects into the virtual braid group over the same number of strands.

\end{abstract}

%\tableofcontents

\section{Introduction}

The virtual braid group is obtained from the classical braid group by adding generators called virtual crossings, and relations. It is easy to see how two braid words which contain only classical generators and represent the same element in the classical braid group need to represent the same element of the virtual braid group on the same number of strand. The more complicated question is to show that two braid words over classical generators which represent the same element of the virtual braid group also represent the same element in the classical braid group. This claim can be proved through the representation of braids in the automorphism group of free groups, as first noted by Kamada in \cite{Kam2004} where he relies on the surjection from the virtual braid group to the titular braid-permutation group from \cite{FRR1997}. That technique, while beautiful is far from elementary. This paper proposes a solution which gives an explicit map, the Gaussian projection, from a sequence of virtual moves between classical braid diagrams to a sequence of classical moves between classical braid diagrams. \\

This paper is structured as follows: Section \ref{groups} defines the classical and virtual braid groups, Section \ref{alex} defines almost classical braids and gives the conditions for almost classical braids to be classical braids, Section \ref{parity} constructs the Gaussian projection, and shows that it maps virtual braid diagrams to almost classical braid diagrams, and Section \ref{mainthm} provides a proof that classical braid groups inject in virtual braid groups.

\section{Virtual braid groups and diagrams} \label{groups}

This section contains a review of the definitions of the classical and virtual braid groups. In order to use the tool of Gaussian projection, explored in the next section, it is important to understand braids as diagrams in the plane, a translation to this system from the algebraic interpretation is provided below.

\subsection{Classical braid groups}

Let $n\in \N$. The \emph{classical braid group} on $n$ strands, denoted $B_n$, is the group of words over the elements $\sigma_1, \sigma_2, \ldots , \sigma_{n-1}$ and their formal inverses, subject to the relations:
$$ \mathbf{(U1)} \ \ \forall \ i \textrm{ and } j, \textrm{ if } \ |i-j|>1 \textrm{ then } \sigma_i\sigma_j=\sigma_j\sigma_i, $$
$$ \mathbf{(U2)} \ \ \sigma_i \sigma_i\-= \sigma_i\- \sigma_i=1,$$
$$ \mathbf{(U3)} \ \ \sigma_i\sigma_{i-1}\sigma_i= \sigma_{i-1}\sigma_i\sigma_{i-1},$$
where $i$ and $j$ are in $\{1, 2, \ldots, n-1\}$. Astute readers will notice that U2 follows from the definition of $B_n$ as a group, and its inclusion as a relation is solely for the purpose of introducing the notation.

An element of $B_n$ is called a \emph{braid}, and a finite presentation of this element is called a \emph{braid word}. Geometrically, braid words can be seen as diagrams by assigning to each generating element a crossing as done in Figure \ref{classcross} and stacking them from the bottom to the top as the word is read from left to right. 

%The \emph{classical pure braid group}, $P_n$ is the subgroup of $B_n$ consisting of elements in the kernel of the map from $B_n$ to the symmetric group $S_n$ which sends $\sigma_i^{\pm1}$ to the permutation $(i \ i+1)$. 

\begin{figure}[htbp]
\centering
\def \svgwidth{\textwidth} 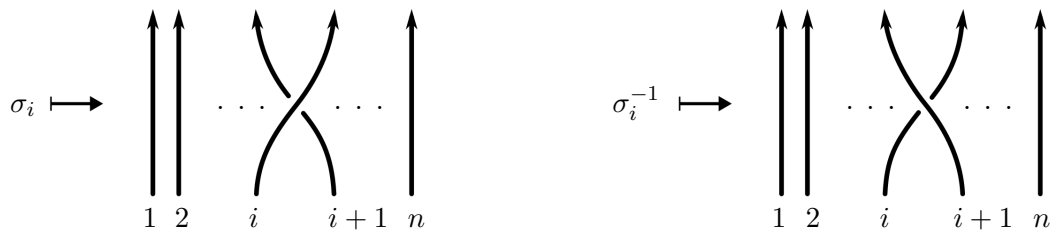  
\caption{The realization of the classical braid group generators as braid diagrams. } \label{classcross}
\end{figure}

The relations on braid words are translated to braid diagram moves by realizing the equivalent words as diagrams. Modifying a braid diagram by replacing a part of it by an equivalent local depiction is called a \emph{classical braid move}, which in this paper are classified to be of either type U2 or U3. See Figure \ref{braidmoves} for examples of a move of type U2 and a move of type U3. Moves of type U1 are ignored as they can be considered to the result of isotopies of the region in which the braid diagram is drawn. In addition to the move given by the U3 relation, all moves obtained from it by mirror symmetry of the diagrams (followed by switching the orientation of the strands to be upwards if needed) are also said to be of that same type. 

\begin{figure}[htbp] 
\centering
\def \svgwidth{\textwidth} 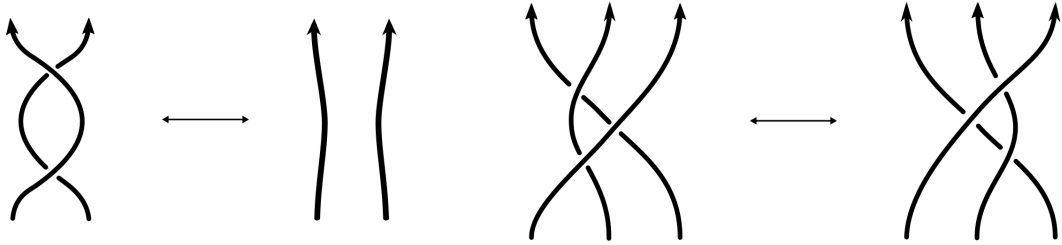  
\caption{Two of the classical braid moves.} \label{braidmoves}
\end{figure}

\subsection{Virtual braid groups}

The \emph{virtual braid group} on $n$ strands, denoted $vB_n$, is a generalization of the classical braid group  first mentioned in \cite{Kau1999} as a topic for further research. The group consists of the words over the elements  $\sigma_1, \sigma_2, \ldots , \sigma_{n-1}, \tau_1, \tau_2, \ldots , \tau_{n-1}$, subject to the relations U1, U2, and U3 above along with the following:
$$\mathbf{(V1)} \ \  \forall \ i \textrm{ and } j, \textrm{ if } \ |i-j|>1 \textrm{ then } \tau_i\tau_j=\tau_j\tau_i,$$
$$\mathbf{(V2)} \ \ \tau_i \tau_i=1,$$
$$\mathbf{(V3)} \ \ \tau_i\tau_{i-1}\tau_i= \tau_{i-1}\tau_i\tau_{i-1},$$ 
$$\mathbf{(V4)}\ \  \tau_i\tau_{i-1}\sigma_i= \sigma_{i-1}\tau_i\tau_{i-1},$$
$$\mathbf{(V5)} \ \   \forall \ j, \ |i-j|>1,\tau_i\sigma_j=\sigma_j\tau_i,$$
where $i$ and $j$ are in $\{1, 2, \ldots, n-1\}$. 

An element of $vB_n$ is called a \emph{virtual braid}, and like the classical braids, it can be represented by diagrams, by adding the representation of the virtual crossings in Figure \ref{virtcross}. 

\begin{figure}[htbp]
\centering
\def \svgwidth{\textwidth} 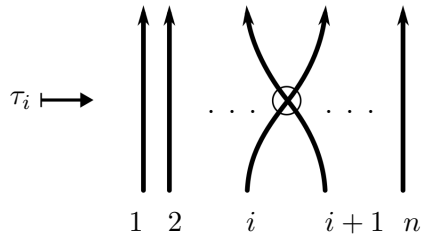  
\caption{The realization of the virtual generators as braid diagrams.} \label{virtcross}
\end{figure}

The relations on virtual braid words which includes virtual crossings (V1, V2, V3, V4 and V5)  can be generalized and unified into an equivalence of virtual braid diagrams called the \emph{virtual detour move} in \cite{Kau1999}. These moves are all invisible to the classical crossings of the diagram and, for reasons that are beyond the scope of this paper, classical crossings encode all of the information of a virtual braid. Therefore, the impact of virtual detour move on the constructions in the sections below is trivial. The denomination ``virtual detour move'' will therefore be used to shorten arguments. For reference, the moves in Figure \ref{detour} are virtual detour moves corresponding to the relations V2 and V4 respectively. %maybe move this paragraph to a place closed to where I'm using them. 

\begin{figure}[htbp] 
\centering
\def \svgwidth{\textwidth} 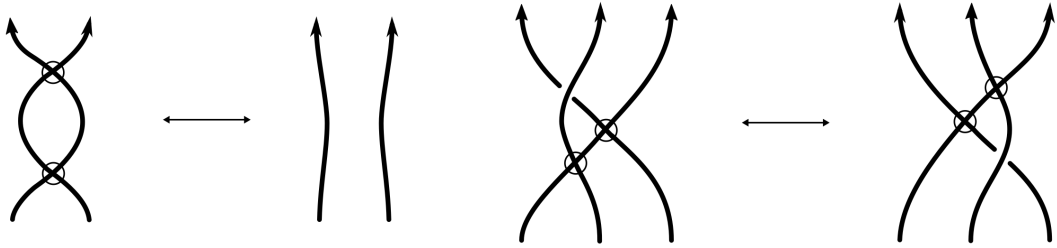  
\caption{Two of the moves on braid diagram involving virtual crossings. } \label{detour}
\end{figure}

%The \emph{virtual pure braid group}, $vP_n$ is the subgroup of $vB_n$ consisting of elements in the kernel of the homomorphism from $vB_n$ to the symmetric group $S_n$ generated by sending $\sigma_i^{\pm1}$ and $\tau_i$ to the permutation $(i \ i+1)$. 

\subsection{Inclusion maps}

It is clear that, for each $n=2, 3, \ldots$, there is an inclusion map from the set of classical braid diagrams on $n$ strands to the set of virtual braid diagrams on $n$ strands, defined as the identity of the generators.  %its restriction $v:P_n \to vP_n$ is denoted by the same letter with no ambiguity. 
This paper aims to show, in Section \ref{mainthm}, that the maps $\bar v: B_n\to vB_n$, induced by $v$, are injective for all $n\in \N$. 

\section{Alexander numberings} \label{alex}

For a classical oriented link diagram, the \emph{Alexander numbering} of the complement of the diagram is a signed measure of how far a component is from being the region at infinity. The concept has been adapted to other settings by first pushing the numbering from the complement of the diagram onto the diagram itself, which allowed to extend it to virtual oriented link diagrams. This section defines the virtual braid diagrams which have Alexander numbers to be \emph{almost classical} and proves some of their properties. %reference

\subsection{General definition}
Let $\beta$ be a virtual or classical braid diagram. The \emph{strands} of $\beta$ are the oriented curves running from the bottom to the top of the braid. The boundary of these curves is called the \emph{endpoints} of the braids. The \emph{arcs} of $\beta$ are the connected components of the strands with all the classical crossings removed. Each arc connects either two classical crossings, two endpoints, or an endpoint and a classical crossing, and may contain virtual crossings. 

An \emph{integer numbering} of a classical braid diagram is an assignment of an integer to each arc of the diagram, such that it satisfies, at each crossing, the relation depicted in Figure \ref{crossnum}. If moreover, $\lambda=\mu-1$ at each crossing of the diagram, the numbering is called an \emph{Alexander numbering}. It is sometimes convenient to see the local condition of the Alexander numberings as saying that the arcs are numbered such that the operation seen in Figure \ref{smooth}, called the \emph{oriented smoothing}, is always merging together arcs with equal numbers. 

\begin{figure}[htbp]
\centering
\def \svgwidth{\textwidth} 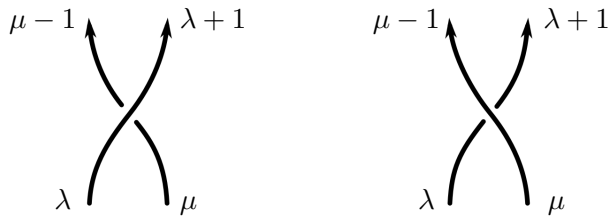  
\caption{Numbering around a crossing.} \label{crossnum}
\end{figure}

\begin{prop} \label{classicalnum}
Any classical braid diagram admits an Alexander numbering with the boundary conditions that the $i$th strand's first arc be numbered $i$, and that the $i$th strand from the left at the top of the braid have its last arc also be numbered $i$. (See Figure \ref{boundnum} for a diagram with the boundary conditions on the bottom of the strands.)
\end{prop}

\begin{proof}
Let $\beta$ be a classical braid diagram. Its first crossing from the bottom is $\sigma_\lambda^{\pm 1}$ and with the assigned boundary condition, its arcs are numbered as seen on the appropriate crossing of Figure \ref{smooth}. By smoothing this crossing, the same argument shows that the whole diagram is numberable.
\end{proof}

In fact, this proof can also be used to show that assuming the local conditions of Alexander numbering and the boundary condition on only either the top or the bottom of the braid to be holding implies that the boundary condition on the other end of the braid also holds. 

\begin{figure}[htbp] 
\centering
\def \svgwidth{\textwidth} 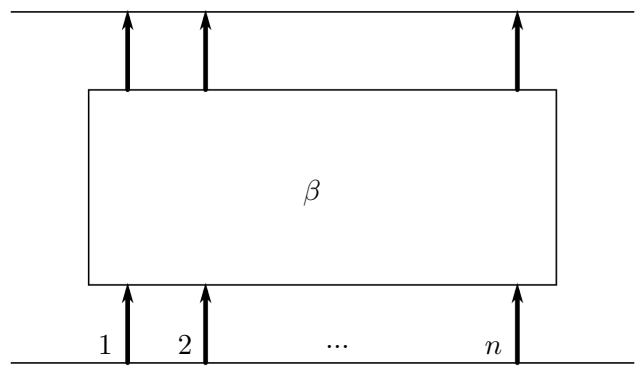  
\caption{Boundary conditions on the Alexander numbering for a braid diagram. } \label{boundnum}
\end{figure}

\subsection{Numbering virtual braid diagrams} 

It is now possible to define integer numberings of virtual braid diagrams, to be an assignment of an integer to each arc of the diagram, where each arc is delimited by classical crossings and by endpoints, and the numbering increases or decreases at each classical crossing as seen in Figure \ref{crossnum}. Again, an integer numbering is called an Alexander numbering if and only if it satisfies the boundary conditions in Figure \ref{boundnum} and the local Alexander numbering condition as seen in Figure \ref{smooth} for each classical crossing.

It is quite simple to find that there are virtual braid words for which the corresponding braid diagram has no Alexander numbering. One such example is seen in Figure \ref{example1}, which depicts a diagram for the word the word $\tau_1\sigma_1\tau_1\sigma_1\in vB_2$. The first classical crossing in the braid satisfies the Alexander numbering condition if and only if $i=j-1$, while the second classical crossing requires that $j+1=i-2$. It is impossible for both of those equations to hold simultaneously. This gap between virtual braid diagrams and virtual braid diagrams which admit Alexander numberings justifies distinguishing between integer and Alexander numberings in the previous subsection. 

\begin{figure}[htbp] 
\centering
\def \svgwidth{\textwidth} 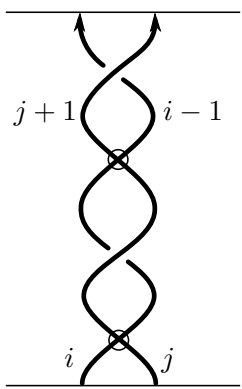  
\caption{A braid which cannot admit an Alexander numbering. } \label{example1}
\end{figure}

Since all classical braid diagrams have Alexander numberings, we say the following:

\begin{definition}
A virtual braid diagram is \emph{almost classical} if it can be given an integer numbering that satisfies the restrictions in Figures \ref{crossnum} with $\lambda=\mu-1$ and \ref{boundnum}.
A virtual braid is almost classical if it is representable by an almost classical virtual braid diagram. 
\end{definition}

\subsection{Almost classical braids}
It is inconvenient that the concatenation of two almost classical virtual braid diagrams may itself not be almost classical. One such example is the diagram representing the virtual braid word $\tau_1\sigma_1\in vB_2$ decomposes into two almost classical braid diagrams, with one (respectively virtual and classical) crossing each. A solution would be to ask for the stronger boundary conditions, where the numbers of the arcs at the top of the braid are also consecutively increasing integers, starting with $1$ on the left. This solution would make those particularly well-numbered virtual braids into a subgroup of the virtual braid group, and satisfy the following lemma:

\begin{lemma} \label{acisu}
Let $\beta$ be a virtual braid diagram on $n$ strands. If $\beta$ is almost classical and the top endpoints of $\beta$ are numbered $1$ through $n$ consecutively from left to right, then all the virtual crossings of $\beta$ can be removed by a finite sequence of virtual detour moves.
\end{lemma}

%Notice that arcs are not broken by virtual crossings, hence they do not interact with the numbering.
%The quality of almost classicality for a virtual braid diagram is somewhat robust, as illustrated by the interaction between the numbering and classical braid diagram move. 

\begin{figure}[htbp] 
\centering
\def \svgwidth{\textwidth} 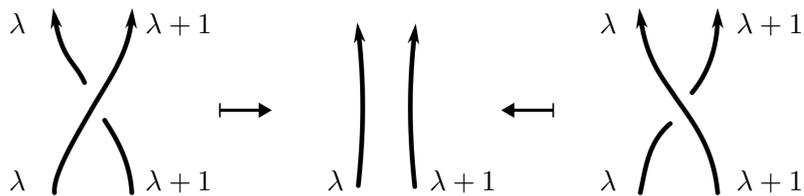  
\caption{The Alexander numbering of crossings and their smoothing.} \label{smooth}
\end{figure}

Lemma \ref{acisu} can be restated to say that almost classical virtual braids with a numbering that satisfies the boundary condition on the top and the bottom are in the image of $\bar v: B_n \to vB_n$. This justifies the laxer boundary condition and the inconvenience of almost classical braids on $n$ strands not being a subgroup of $vB_n$. 

\begin{proof}
If $\beta$ has no classical crossings, the only restriction on the Alexander numbering comes from the endpoints of each strand. The definition dictates that the $i$th strand of the diagram be numbered $i$ at the top and the bottom, for each $i=1,2, \ldots, n$. Since $\beta$ contains no classical crossings, each strand consists of a single arc, and the numbering shows that the underlying permutation is trivial. Therefore, such a $\beta$ is equivalent to the identity braid. 

Now, assume that the lemma holds for any braid diagram with $k$ classical crossings for some natural number number $k$. Let $\beta$ be an almost classical braid diagram with $k+1$ classical crossings. Then, the diagram $\beta_0$ obtained by smoothing any of the classical crossings of $\beta$ as in Figure \ref{smooth} is also almost classical and has $k$ classical crossings. By our assumption, $\beta_0$ is related to a classical braid diagram by a finite sequence of virtual detour moves, and by the definition of those moves, the same sequence takes $\beta$ to a classical braid diagram. 
\end{proof}

\section{Parity projection} \label{parity}

This section defines a map from virtual braids to almost classical braids, which will be used to prove the main theorem of this paper.

\subsection{Gaussian parity}
In general, a virtual braid diagram is not Alexander numberable. Its arcs can still be numbered by using the rule that the $i$th strand starts with number $i$, and that the numbering changes by $+1$ or $-1$ at each classical crossing as depicted in Figure \ref{crossnum}. The result is again called an integer numbering of the dia. For some crossings, the numbering will still respect the Alexander numbering rule that the incoming and outgoing arcs on each side of the crossing bear the same number. In general, they will not. Following the convention established in \cite{IMN2013} say that a crossing in a braid is \emph{even} if and only if the integer numbering of the arcs respects the Alexander condition around that crossing. Otherwise, the crossing is called \emph{odd}. The property that a crossing be odd or even is its \emph{parity}. 

For example, consider the two crossings in Figure \ref{proofrm2}, considered to be portions of the same, larger braid diagram. Both crossings are even whenever $\lambda= \mu-1$, and both are odd otherwise. Moreover, observe that the integer numbering of the rest of the diagram is unchanged by the local move.

\begin{figure}[htbp]
\centering
\def \svgwidth{\textwidth} 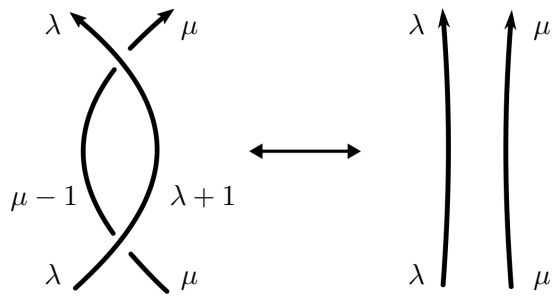  
\caption{The Alexander numbering around a canceling pair of classical crossings in a braid diagram. } \label{proofrm2}
\end{figure}

In Figure \ref{proofrm3}, the crossings on the right and the left version can be identified by the numbering of the strands at the bottom of the picture. The $\gamma/\lambda$ crossing is even if and only if $\gamma= \lambda-1$ (on the left picture, the condition is equivalent, and reads $\gamma+1= \lambda$). Similarly, the parity of the $\gamma/\mu$ and the $\lambda/\mu$ crossings is unchanged by the move, and by looking at the labels at the top of the picture, it still holds that the integer numbering of the rest of the picture is unchanged by the local move. The more interesting property that can be computed in this situation is that if any two of the crossings pictured are even, so is the third one. That is, if $\gamma= \lambda-1$ and $\gamma+ 1= \mu-1$, then the $\lambda/\mu$ crossing is also even. Similarly for the other two pairs of crossings. 

\hfill

\begin{figure}[htbp] 
\centering
\def \svgwidth{\textwidth} 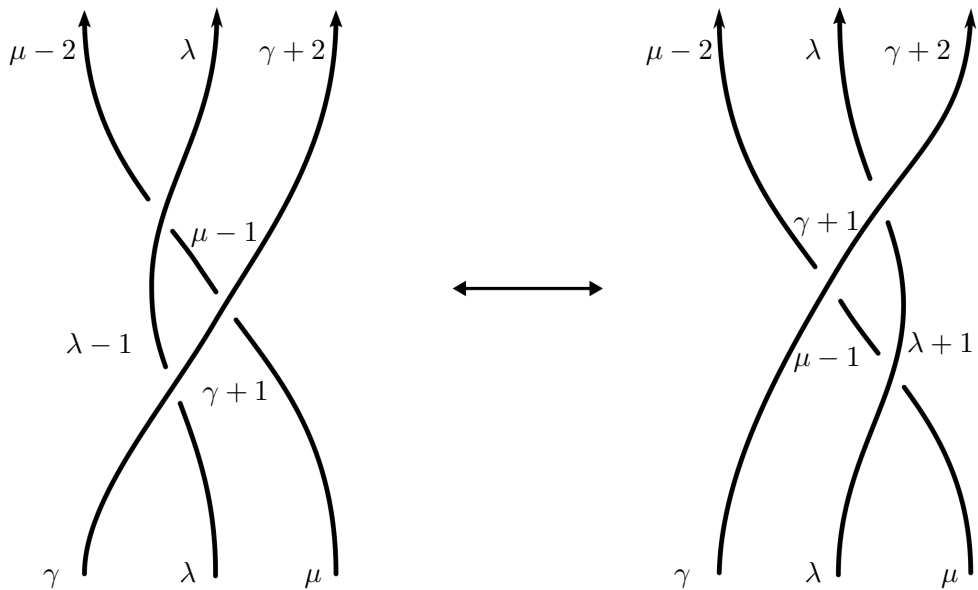  
\caption{The Alexander numbering around three crossings before and after a move on a braid diagram. } \label{proofrm3}
\end{figure}

\subsection{Projection map}

Define the \emph{Gaussian projection} of a virtual braid diagram $\beta$ to be $\phi(\beta)$, the braid diagram obtained by turning all the odd crossings of $\beta$ into virtual crossings, computing the parity of the remaining classical crossings with the new arcs, turning any new odd crossings virtual, and repeating this process until the resulting braid diagram is almost classical. 

\begin{lemma} \label{lem2}
Let $\beta$ and $\beta'$ be equivalent virtual braid diagrams. Then, $\phi(\beta)$ and $\phi(\beta')$ are equivalent almost classical braid diagrams.
\end{lemma}

\begin{proof} 
First note that the assertion that $\phi(\beta)$ and $\phi(\beta')$ are almost classical pure braid diagrams follows trivially from the definition of the map $\phi$.
Now, if $\beta'$ is obtained from $\beta$ by a virtual detour move, there is a one-to-one correspondence between the classical crossings of each diagram, and the parity of the crossings is not impacted by virtual detour moves. Thus, $\phi(\beta')$ is obtained from $\phi(\beta)$ by a virtual detour move. 

Now, assume that $\beta$ and $\beta'$ differ by only one classical braid move. Considering first a move of type U2, where $\beta$ is the diagram with more crossings. Then, those crossings are either both preserved by $\phi$, or both made virtual by $\phi$. In either case, the diagram $\phi(\beta')$ can be obtained from $\phi(\beta)$. An example or this situation is shown in Figure \ref{mainfig}. 

Now, if $\beta'$ is obtained from doing a move of type U3 to $\beta$, there are three possible situation. The crossings involved in the move can be all odd, all even, or only one of the three could be even. Again, the number of odd and even crossings in $\beta$ and $\beta'$ is equal, and, should exactly one of them be even, it is the crossing between the corresponding stands in each diagram. Then, $\phi(\beta')$ can be obtained from $\phi(\beta)$ by applying either a V3, a U3, or a virtual detour move. See Figure \ref{mainfig} for an example.

It suffices to apply the arguments above to each move in a sequence to obtain the general statement.\end{proof}

In other terms, Lemma \ref{lem2} states that the map of braid diagrams, $\phi$, factors to a map of virtual braids, $\bar \phi: vB_n \to vB_n$, and that a braid is almost classical if and only if it is in the image of $\bar \phi$. Notice that almost classical braids are a proper subset of $vB_n$ for each $n$, but that they do not form subgroups. 

\begin{figure}[htbp] \centering
\centering
\def \svgwidth{\textwidth} 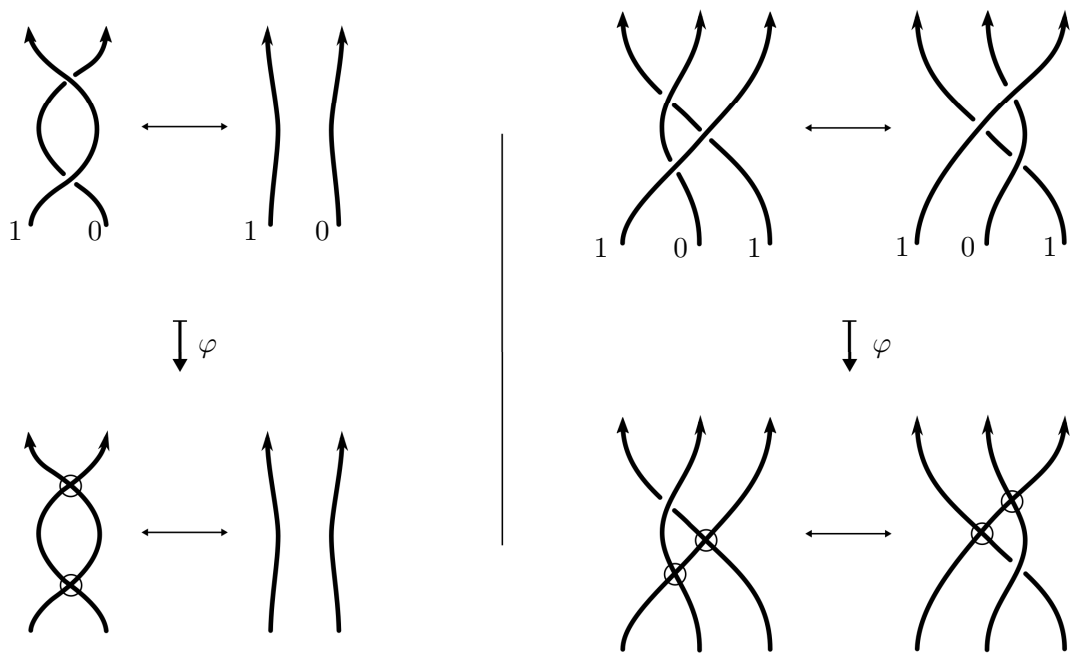  
\caption{The result of the Gaussian projection map on some braid moves.} \label{mainfig}
\end{figure}

%\subsection{Projection is a homomorphism for pure virtual braids}

%\begin{prop} Parity projections map pure virtual braids to pure virtual braids. \end{prop}

%The problem with using $vB_n$ is that the image of $\phi(vB_n)$ contains braids that consists of only virtual crossings and thus are not in $v(B_n)$, so it's cleaner to use the pure braid groups. 

\section{Main theorem} \label{mainthm}

\begin{thm} For each $n=2, 3, \ldots$, the map $\bar v: B_n \to vB_n$ is injective.
\end{thm}

% $v \circ \phi= id$, thus $v$ is 1-1. 

\begin{proof} 
Fix $n$ to be an integer greater than 1, and let $\beta$ and $\beta'$ be classical braid diagrams representing the same element in $vB_n$. To prove that $v$ is injective, it suffices to show that $\beta$ and $\beta'$ represent the same element in the classical pure braid group $B_n$. 
Let $\{\beta=\beta_0, \beta_1, \ldots , \beta_k=\beta'\}$ be a sequence of virtual braid diagrams such that $\beta_{i-1}$ and $\beta_i$  for each $i=1,\ldots k$ differ by exactly one classical braid move and possibly a virtual detour move (defined in the discussion of Figure \ref{detour}). Then, the sequence $\{\phi(\beta_0)= \beta, \phi(\beta_1), \ldots , \phi(\beta_k)=\beta'\}$ consists of almost classical braid diagrams, and applying Lemma \ref{lem2}, $\phi(\beta_i)$ is obtained from $\phi(\beta_{i-1})$ by virtual detour moves and at most one classical braid move.

Since $\beta$ is classical, the numbering at the top $\phi(\beta_1)$ is consecutively increasing from 1. Then by Lemma \ref{acisu}, the virtual crossings $\phi(\beta_1)$  can be removed by virtual detour moves, to obtain a classical braid diagram. This last argument can be applied to each projected diagram in turn, yielding an explicit sequence of classical braid diagram moves relating $\beta$ and $\beta'$. 
\end{proof}

A part of the proof technique in action is illustrated by a rather artificial example in Figure \ref{example}. Say that the diagram on the left, corresponding to $\tau_1\sigma_2\sigma_2\-\tau_1\in vB_3$ appears at some point in a sequence of virtual braid diagrams. Then, the projection maps it to the almost classical diagram of the trivial braid in $vB_3$ seen on the right of the figure.

\begin{figure}[htbp] \centering
\centering
\def \svgwidth{\textwidth} 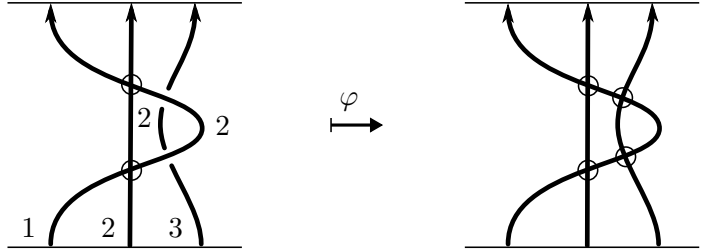  
\caption{Example of a trivial braid projecting to a trivial almost classical braid.} \label{example}
\end{figure}

%When attempting to extend this technique to work for the general braid group and virtual braid group, almost classical braids fail to be a subgroup. Another interesting avenue of generalization is given by virtual string links. 

\end{document}